\documentclass[10pt]{amsart}
\usepackage{amsmath,amscd}
\usepackage{amsbsy}
\usepackage{amssymb}
\usepackage{amscd,amsthm}
\usepackage[all,cmtip]{xy}
\usepackage[russian,english]{babel}
\usepackage[utf8]{inputenc}
\usepackage{xcolor}

\usepackage{hyperref}

\renewcommand{\epsilon}{\varepsilon}

\renewcommand{\ell}{x}

\newtheorem{thm}{Theorem}\numberwithin{thm}{section}
\newtheorem{lem}[thm]{Lemma}

\begin{document}
\begin{center}
\huge{Diophantine equations involving powers of factorials}\\[1cm] 
\end{center}
\begin{center}

\large{Sa$\mathrm{\check{s}}$a Novakovi$\mathrm{\acute{c}}$}\\[0,5cm]
{\small September 2025}\\[0,5cm]
\end{center}
{\small \textbf{Abstract}. 
We are motivated by a result of Alzer and Luca who presented all the integer solutions to the relations $(k!)^n-k^n=(n!)^k-n^k$ and $(k!)^n+k^n=(n!)^k+n^k$. We consider the equations $(k!)^{n!}\pm k^n=(n!)^{k!}\pm n^k$ and $(k!)^n\pm k^{n!}=(n!)^k\pm n^{k!}$ and prove a similar statement.}


\begin{center}
\tableofcontents
\end{center}
\section{Introduction}
The theory of Diophantine equations has a long and rich history and has attracted the attention of many mathematicians. In particular, the study of diophantine equations involving factorials have been studied extensively. For example Brocard \cite{BR}, and independently Ramanujan \cite{RA}, asked to find all integer solutions for $n!=x^2-1$. It is still an open problem, known as Brocard's problem, and it is believed that the equation has only three solutions $(x,n)=(5,4), (11,5)$ and $(71,7)$. Overholt \cite{O} observed that a weak form of Szpiro's-conjecture implies that Brocard's equation has finitely many integer solutions. Some further examples of similar equations are:
\begin{itemize}
	\item[1)] $n!=x^k\pm y^k$ and $n!\pm m!=x^k$, see \cite{EO}.
	\item[2)] $\phi(x)=n!$, where $\phi$ is the Euler totient function \cite{FL}.
	\item[3)] $p(x)=m!$, where $p(x)\in\mathbb{Z}[x]$ \cite{L}.
	\item[4)] $\alpha\,m_1!_{S_1}\cdots m_r!_{S_r}=f(n!)$, where $f$ is an arithmetic function and $m_i!_{S_i}$ are certain Bhargava factorials \cite{BN}.
\end{itemize}
For the equations 1) and 4), it was shown that the number of integer solutions is finite. The equation in 3) has finitely many integer solutions, provided the ABC-conjecture holds, and 2) does have infinitely many solutions. There are a lot of more diophantine equations involving factorials and ploynomials that have been studied and we refer the interested reader to \cite{BN}, \cite{NO} and the references therein. For example, Alzer and Luca \cite{AL} considered the equations  $(k!)^n-k^n=(n!)^k-n^k$ and $(k!)^n+k^n=(n!)^k+n^k$ and presented all the integer solutions. Their results and methods motivated us to study the following diophantine relations:
$$
(k!)^{n!}\pm k^n=(n!)^{k!}\pm n^k \quad \textnormal{and} \quad (k!)^n\pm k^{n!}=(n!)^k\pm n^{k!}.
$$
Our results are the following theorems.

\begin{thm}
	Let $n$ and $k$ be positive integers. The equation
	$$
	(k!)^{n!}-k^n=(n!)^{k!}-n^k
	$$
	holds if and only if $k=n$ or $(k,n)=(1,2), (2,1)$.
\end{thm}

\begin{thm}
	Let $n$ and $k$ be positive integers. The equation
	$$
	(k!)^{n!}+k^n=(n!)^{k!}+n^k
	$$
	holds if and only if $k=n$.
\end{thm}

\begin{thm}
Let $n$ and $k$ be positive integers. The equation
$$
(k!)^n-k^{n!}=(n!)^k-n^{k!}
$$
holds if and only if $k=n$ or $(k,n)=(1,2), (2,1)$.	
\end{thm}

 \begin{thm}
 	Let $n$ and $k$ be positive integers. The equation
 	$$
 (k!)^n+k^{n!}=(n!)^k+n^{k!}	
 	$$
 	if and only if $k=n$.	
 \end{thm}
\noindent
We want to make a short comment regarding the proofs of the above assertions. Theorems 1.2 and 1.3 are proved using minor adoptations of the arguments in the proofs of the results of Alzer and Luca \cite{AL}. But it turned out that the proofs of Theorems 1.1 and 1.4 needed a double induction approach and was not so straight forward as we thought. We are not sure whether the proofs can be simplified. 
\section{Proof of Theorem 1.1}
\noindent
For the proof we need the following two lemmas.
\begin{lem}
	For $n>2$ one has $2^{n!}-2^n>n!^2$.
\end{lem}
\begin{proof}
	We first show by induction that
	$$
	\frac{2^{(n-1)!}}{n!^2}>2
	$$
	for $n\geq 5$. One verifies directly that the inequality holds for $n=5$. It follows
	$$
	2^{n!}=(2^{(n-1)!})^n>(2(n!)^2)^n=2^n((n!)^2)^n>2((n!)^2)^n>2(n+1)!^2
	$$	
	and this shows
	$$
	2^{(n-1)!}>2(n!)^2.
	$$
	This implies
	$$
	2^{(n-1)!}-n!^2> n!^2>1
	$$
	and therefore
	$$
	2^{(n-1)!}-1>n!^2.
	$$
	And since $2^n(2^{(n-1)!}-1)=2^{n!}-2^n$, we find for $n\geq 5$
	$$
	2^{n!}-2^n>n!^2.
	$$
	Now checking the cases $n=3$ and $n=4$ completes the proof of the lemma.
\end{proof}
\begin{lem}
	For $k\geq 3$ one has $k^{(k+1)!}>(k+1)!^k+(k+1)^{k!}$.
\end{lem}
\begin{proof}
	One easily verifies that $3^{4!}>4!^3+4^{3!}$. We now proceed by induction on $k$. For this we must show
	$$
	(k+1)^{(k+2)!}>(k+2)!^{(k+1)}+(k+2)^{(k+1)!}.
	$$
	Indeed, by induction hypothesis we find
	$$
	(k+1)^{(k+2)!}=((k+1)^{(k+1)!})^{(k+2)}>(k^{(k+1)!})^{(k+2)}>((k+1)!^k+(k+1)^{k!})^{(k+2)}.
	$$
	Notice that 
	$$
	((k+1)!^k+(k+1)^{k!})^{(k+2)}=\sum_{j=0}^{k+2}\binom{k+2}{j}(k+1)!^{k(k+2-j)}(k+1)^{k!j}
	$$
	and 
	$$
	\sum_{j=0}^{k+2}\binom{k+2}{j}(k+1)!^{k(k+2-j)}(k+1)^{k!j}>(k+2)(k+1)!^k(k+1)^{k!(k+1)}+(k+1)^{k!(k+2)}.
	$$
	This yields
	$$
	(k+1)^{(k+2)!}>(k+2)(k+1)!^k(k+1)^{k!(k+1)}+(k+1)^{k!(k+2)}.
	$$
	We now show $(k+1)^{k!(k+2)}>(k+2)^{(k+1)!}$ and $(k+2)(k+1)!^k(k+1)^{k!(k+1)}>(k+2)!^{(k+1)}$. Let us consider
	$$
	(k+1)^{k!(k+2)}=((k+1)^{k+2})^{k!}
	$$
	and 
	$$
	(k+2)^{(k+1)!}=((k+2)^{k+1})^{k!}.
	$$
	Since $k+1\geq 4$ and $(1+1/(k+1))^{k+1}$ is strictly increasing with limit being $e$, we conclude
	$$
	k+1>(1+\frac{1}{k+1})^{k+1}=(\frac{k+2}{k+1})^{k+1}.
	$$
	This shows
	$$
	(k+1)^{k+2}>(k+2)^{k+1}
	$$
	and hence
	$$
	((k+1)^{k+2})^{k!}>((k+2)^{k+1})^{k!}.
	$$
	Now we show $(k+2)(k+1)!^k(k+1)^{k!(k+1)}>(k+2)!^{(k+1)}$. Indeed, since $(k+1)^{k+1}>(k-j)(k+2)$ for $j=0,...,(k-1)$ and $k\geq 3$, we conclude 
	$$
	((k+1)^{k+1})^k>k!(k+2)^k
	$$
	and therefore
	$$
	((k+1)^{k+1})^{k!}>(k+1)k!(k+2)^k.
	$$
	Hence
	$$
	(k+1)^{k!(k+1)}>(k+1)!(k+2)^k.
	$$
	This gives
	$$
	(k+2)(k+1)^{k!(k+1)}>(k+2)(k+1)!(k+2)^{k}
	$$
	and finally
	$$
	(k+2)(k+1)!^k(k+1)^{k!(k+1)}>(k+2)(k+1)!^{k+1}(k+2)^{k}=(k+2)!^{k+1}.
	$$
	This proves that $k^{(k+1)!}>(k+1)!^k+(k+1)^{k!}$ for $k\geq 3$.
\end{proof}
\noindent
\textbf{proof of Theorem 1.1}:\\
\noindent
Obviously, if $k=n$ or $(k,n)=(1,2),(2,1)$ the equation $(k!)^{n!}-k^n=(n!)^{k!}-n^k$ is valid. Next, we show that if the equation holds with $k<n$, then $(k,n)=(1,2)$. We distinguish three cases.\\
\noindent
\emph{the case $k=1$}: \\
\noindent
the equation becomes 
$$
1-1=0=n!-n=n((n-1)!-1).
$$
Since $n>1$, we obtain $n=2$.\\  
\noindent
\emph{the case $k=2$}: \\
\noindent
Now the equation reduces to
$$
2^{n!}-2^n=n!^2-n^2.
$$

\noindent
Using Lemma 2.1, we find 
$$
2^{n!}-2^n>n!^2-n^2
$$
for $n>2$. Therfore, the equation $2^{n!}-2^n=n!^2-n^2$ has no solution for $n>2$.\\  
\noindent
\emph{the case $k=3$}: \\
\noindent 
We show that $k!^{n!}>n!^{k!}+k^n$ for $n>k\geq 3$. We use a double induction argument. To realize the induction start, we have to show that 
$$
k!^{(k+1)!}>(k+1)!^{k!}+k^{k+1}
$$
for $k\geq 3$. From Lemma 2.2 we know $k^{(k+1)!}>(k+1)!^k+(k+1)^{k!}$ for $k\geq 3$. Multiplying the inequality by $(k-1)!^{(k+1)!}$ gives
$$
k!^{(k+1)!}>(k-1)!^{(k+1)!}(k+1)!^k+(k-1)!^{(k+1)!}(k+1)^{k!}.
$$ 
Now we show $(k-1)!^{(k+1)!}(k+1)^{k!}>(k+1)!^{k!}$ and $(k-1)!^{(k+1)!}(k+1)!^k>k^{k+1}$. We start by proving the first inequality. Indeed, since $(k-1)!^k>k$ for $k\geq 3$, we find
$$
(k-1)!^{k+1}>k!
$$
and hence
$$
(k-1)!^{(k+1)!}>k!^{k!}.
$$
But this yields
$$
(k-1)!^{(k+1)!}(k+1)^{k!}>(k+1)!^{k!}.
$$
To show that $(k-1)!^{(k+1)!}(k+1)!^k>k^{k+1}$, we note that $(k+1)^k>k$ for $k\geq 3$. But this implies
$$
(k-1)!^{(k+1)!}(k-1)!^k(k+1)^k>k
$$
and therefore
$$
(k-1)!^{(k+1)!}(k+1)!^k>k^{k+1}.
$$
This realizes the induction start. To show that  $k!^{n!}>n!^{k!}+k^n$, we continue using induction on $n>k$. So we fix a $k\geq 3$ and obtain from the induction hypothesis
$$
k!^{(n+1)!}=(k!^{n!})^{n+1}>(n!^{k!}+k^n)^{n+1}=\sum_{j=0}^{n+1}\binom{n+1}{j}n!^{k!(n+1-j)}k^{nj}.
$$
Notice that
$$
\sum_{j=0}^{n+1}\binom{n+1}{j}n!^{k!(n+1-j)}k^{nj}>k^{n(n+1)}+(n+1)n!^{k!n}k^n> k^{n+1}+(n+1)n!^{k!n}k^n.
$$
It remains to show that $(n+1)n!^{k!n}k^n>(n+1)!^{k!}$. But this follows from $n!^{n-1}>(n+1)$. Since $n!^{n-1}>(n+1)$, we conclude
$$
n!^n>(n+1)!
$$
and this gives
$$
n!^{k!n}>(n+1)!^{k!}.
$$
But then obviously $(n+1)n!^{k!n}k^n>(n+1)!^{k!}$. This completes the proof.
\section{Proof of Theorem 1.2}
\noindent
We show that $(k!)^{n!}+k^n=(n!)^{k!}+n^k$ implies $k=n$. By symmetry, we may assume $n\geq k$. We consider again three cases.\\
\noindent
\emph{the case $k=1$}: \\
\noindent 
the equation than becomes 
$$
2=n!+n.
$$
But this gives $n=1$.\\
\noindent
\emph{the case $k=2$}: \\
\noindent
now the equation reduces to
$$
2^{n!}+2^n=n!^2+n^2=n^2((n-1)!^2+1).
$$
We can rewrite the left hand side and obtain
$$
2^n(2^{n!-n}+1)=2^{n!}+2^n=n^2((n-1)!^2+1).
$$
It follows that $n=2^ax$ for some $a\geq 1$ and some odd integer $x\geq 1$. Then we obtain
$$
2^{n-2a}(2^{n!-n}+1)=x^2((n-1)!^2+1).
$$
For $n\geq 3$, the right hand side is allways odd. The left hand side is even, except for the case $n=2a$. But the case $n=2a$ means $2^ax=2a$. And this happens only if $a=1$ and $x=1$. Therefore $n=2$, contradicting the assumption $n\geq 3$. This shows $n=2$.\\  
\noindent
\emph{the case $k\geq 3$}: \\
\noindent 
Since 
$$
k^k\mid (k!)^{n!}, \quad k^k\mid k^n \quad \textnormal{and} \quad k^k\mid (n!)^{k!},
$$
we conclude that $k^k\mid n^k$. This implies $k\mid n$. So let $n=bk$ for some $b\geq 1$. We now assume $b\geq 2$ and produce a contradiction. Since the sequence $\{n^{1/n}\}_{n=3}^{\infty}$ is strictly decreasing, we find $k^n-n^k>0$. Thus
$$
k^n-n^k=(n!)^{k!}-(k!)^{n!}=(n!)^{k!}-(k!)^{(bk)!}>0.
$$
Note that there is a positive integer $c$ such that $(bk)!=c\cdot k!$ and that, since we assumed $b\geq 2$, we have $c>b$. Now we have
$$
(n!)^{k!}>(k!)^{(bk)!}=(k!)^{c\cdot k!}.
$$
It follows $n!>(k!)^c$ and therefore $n!-(k!)^c\geq 1$. Now we get 
$$
k^{bk}=k^n>k^n-n^k=(n!)^{k!}-((k!)^{c})^{k!}=(n!-(k!)^c)\sum_{j=0}^{k!-1}(n!)^j(k!)^{c\cdot (k!-j-1)}>(k!)^{c\cdot (k!-1)}.
$$
And since $k!-1>k-1$ and $c>b$, we obtain
$$
k^{bk}>(k!)^{c\cdot (k!-1)}>(k!)^{b(k-1)}.
$$
This gives $k^{k/(k-1)}>k!$. But this is not true, since for $k\geq 3$ we have
$$
k^{k/(k-1)}\leq k^{3/2}=k\sqrt{k}\leq k(k-1)\leq k!.
$$
Hence $b=1$ and $k=n$. This completes the proof.
\section{Proof of Theorem 1.3}
\noindent
Obviously, if $k=n$ or $(k,n)=(1,2),(2,1)$, then $(k!)^n-k^{n!}=(n!)^k-n^{k!}$ is valid. Next, we show that if the equation holds with $k<n$, then we obtain $(k,n)=(1,2)$. We consider three cases.\\
\noindent
\emph{the case $k=1$}: \\
\noindent 
the equation becomes $1-1=0=n!-n$. Since $n>1$, this implies $n=2$.\\
\noindent
\emph{the case $k=2$}: \\
\noindent 
the equation reduces to
$$
2^n-2^{n!}=n!^2-n^2.
$$
Since $n\geq 3$, the left hand side is allways negative. On the other hand, the right hand side is positive. Hence there is no solution for $k=2$.\\
\noindent
\emph{the case $k\geq 3$}: \\
\noindent 
The sequences $\{(n!)^{1/n}\}_{n=1}^{\infty}$ and $\{-n^{1/n}\}_{n=3}^{\infty}$ are strictly increasing. This gives for $n>k\geq 3$:
$$
(n!)^{k}>(k!)^{n} \quad \textnormal{and} \quad -n^{k!}>-k^{n!}
$$
Adding up the inequalities yields
$$
(n!)^{k}-n^{k!}>(k!)^{n}-k^{n!}.
$$
This completes the proof.
\section{Proof of Theorem 1.4}
\noindent
We show that $(k!)^n+k^{n!}=(n!)^k+n^{k!}$ implies $k=n$. Again, we consider three cases.\\
\noindent
\emph{the case $k=1$}: \\
\noindent 
the equation becomes $1+1=2=n!+n$. This yields $n=1$.\\
\noindent
\emph{the case $k=2$}: \\
\noindent 
the equation becomes
$$
2^n+2^{n!}=n!^2+n^2.
$$
But this equation was treated in the proof of Theorem 1.2 on page 5 and it was shown that $n=2$.\\
\noindent
\emph{the case $k\geq 3$}: \\
\noindent 
We show that $k^{n!}>n!^k+n^{k!}$ for $n>k\geq 3$. We use a double induction argument. To realize the induction start, we first show that for $k\geq 3$ one has
$$
k^{(k+1)!}>(k+1)!^k+(k+1)^{k!}.
$$
But this is exactly the content of Lemma 2.2.

\noindent
To show that $k^{n!}>n!^k+n^{k!}$, we continue using induction on $n>k$. So we fix a $k\geq 3$ and perform the induction step. We have to show that
$$
k^{(n+1)!}>(n+1)!^k+(n+1)^{k!}
$$
From the induction hypothesis we obtain
$$
k^{(n+1)!}=(k^{n!})^{n+1}>(n!^k+n^{k!})^{n+1}=\sum_{j=0}^{n+1}\binom{n+1}{j}n!^{k(n+1-j)}n^{k!j}.
$$
Note that 
$$
\sum_{j=0}^{n+1}\binom{n+1}{j}n!^{k(n+1-j)}n^{k!j}>(n+1)n!^kn^{k!n}+n^{k!(n+1)}.
$$
We show $(n+1)n!^kn^{k!n}>(n+1)!^k$ and $n^{k!(n+1)}>(n+1)^{k!}$. Let us start with the first inequality. Since $n^n>n+1$ for $n\geq 2$, we get
$$
(n^n)^{k!}>(n^n)^{k-1}>(n+1)^{k-1}.
$$
But this implies
$$
(n+1)n^{k!n}>(n+1)^k
$$
and therefore
$$
(n+1)n!^kn^{k!n}>(n+1)!^k.
$$
The second inequality follows easily from the fact that $n^{n+1}>n+1$ for $n\geq 2$. Because this implies directly $(n^{n+1})^{k!}>(n+1)^{k!}$. This completes the proof of $k^{n!}>n!^k+n^{k!}$. But this finally shows $k^{n!}+k!^n>n!^k+n^{k!}$ for $n>k\geq 3$. 

\vspace{0.3cm}
\noindent
{\tiny HOCHSCHULE FRESENIUS UNIVERSITY OF APPLIED SCIENCES 40476 D\"USSELDORF, GERMANY.}\\
E-mail adress: sasa.novakovic@hs-fresenius.de\\


\begin{thebibliography}{999}
	\bibitem{AL} H. Alzer and F. Luca, Diophantine equations involving factorials. Mathematica Bohemica 4 (2017), p. 181-184.
	\bibitem{BN} D. Baczkowski and S. Novakovi\'c, Some diophantine equations involving arithmetic functions and Bhargava factorials. Colloqium Mathematicum, to appear. 
	\bibitem{MIB} M. Bennett, et al., Explicit bounds for primes in arithmetic progressions. Illinois J. Math. (2018).
	\bibitem{BH} M. Bhargava, P-orderings and polynimial functions on arbitrary subsets of Dedekind rings. J. Reine Angew. Math. 490 (1997), 101-127.
	\bibitem{BR} H. Brocard: Question 1532. Nouv. Corresp. Math. 2 (1876); Nouv. Ann. Math. 4 (1885), 391.
	\bibitem{EO} P. Erd\H{o}s and R. Obl\'ath, \"Uber diophantische Gleichungen der Form $n!=x^p \pm y^p$ und $n!\pm m!= x^p$. Acta Szeged. 8 (1937), 241-255.
	\bibitem{FL} K. Ford, F. Luca and C. Pomerance, Common values of the arithmetic function $\phi$ and $\sigma$. Bull. Lond. Math. Soc. 42 (2010), 478-488.
	\bibitem{L} F. Luca, The Diophantine equation $P(x)=n!$ and a result of M. Overholt. Glasnik Matemati\'cki 37 (2002), 269-273.
	\bibitem{NO} S. Novakovi\'c, A note on some polynomial-factorial Diophantine equations. Glasnik Matematicki, to appear.
	\bibitem{O} M. Overholt, The Diophantine equation $n!+1=m^2$. Bull. London. Math. Soc. 42 (1993), 104.
	\bibitem{RA} S. Ramanujan, Question 469. J. Indian Math. Soc. 5 (1913), 59.
	
\end{thebibliography}
\end{document}